\documentclass{amsart}
\usepackage{amssymb, latexsym}
\theoremstyle{plain}
\newtheorem{theorem}{Theorem}
\newtheorem{corollary}{Corollary}
\newtheorem{lemma}{Lemma}

\theoremstyle{definition}

\DeclareMathOperator{\tr}{tr}
\begin{document}
\title[Partial orthogonal spreads]
{Partial orthogonal spreads over $\mathbb{F}_2$ invariant\\ 
 under the symmetric and alternating groups}

\author[Rod Gow]{Rod Gow}
\thanks{School of Mathematics and Statistics,
University College,
Belfield, Dublin 4,
Ireland,
\emph{E-mail address:} \texttt{rod.gow@ucd.ie}}

\begin{abstract}

Let $m\geq 3$ be an integer and let $V$ be a vector space of dimension $2^m$ over  $\mathbb{F}_2$. Let $Q$ be  a non-degenerate
quadratic form  of maximal Witt index $2^{m-1}$ defined on $V$. We show that the symmetric group $\Sigma_{2m+1}$ acts on $V$
as a group of isometries of $Q$ and leaves invariant a
partial orthogonal spread of size $2m+1$. This is enough to show that any group of even order $2m$ or odd order $2m+1$, $m\geq 3$,
acts transitively and regularly on a partial orthogonal spread in $V$. We construct the $\Sigma_{2m+1}$ linear action on $V$ by means
of the spin representation of the symmetric group. Furthermore, when $m\equiv 3\bmod 4$, we show that the partial spread
can be extended by two further maximal totally singular subspaces that are interchanged by $\Sigma_{2m+1}$.
We also show that the alternating group $\mathcal{A}_9$ acts
in a natural manner on a complete spread of size 9 defined on a vector space of dimension 8 over $\mathbb{F}_2$. 
\end{abstract}
\maketitle
\section{Introduction}
\noindent 

\noindent  Let $K$ be a finite field of characteristic 2, with $|K|=q$. Let $V$ be a vector space of even dimension $2r$ over $K$. A quadratic
form defined on $V$ is a function $Q:V\to K$ that satisfies
\begin{eqnarray*}
 Q(\lambda u)&=&\lambda^2Q(u)\\
Q(u+v)&=&Q(u)+Q(v)+ f(u,v)
\end{eqnarray*}
for all $\lambda$ in $K$ and all $u$ and $v$ in $V$. In this formula, $f$ is an alternating bilinear form, called the polarization of $Q$.  
We say that $Q$ is a non-degenerate quadratic form if $f$ is a non-degenerate bilinear form.

We say that a vector $v\in V$ is singular with respect to $Q$ if $Q(v)=0$. We say that a subspace $U$ of $V$ is totally isotropic with respect to $f$ if
$f(u,w)=0$ for all $u$ and $w$ in $U$. We also say that $U$ is totally singular with respect to $Q$ if all the elements of $U$ are singular.
It should be clear that if $U$ is totally singular with respect to $Q$, then it is totally isotropic with respect to $f$, but the converse is not
necessarily true.

We shall assume henceforth that $Q$ is non-degenerate. In this case, the maximum dimension of a totally singular subspace of $V$ is $r$. We say that
$Q$ has index $r$ if there is at least one $r$-dimensional totally singular subspace of $V$. We shall
also  assume henceforth that $Q$ has index $r$.

We recall that an isometry of $Q$ is a $K$-linear automorphism, $\sigma$, say, of $V$ that satisfies
\[
 Q(\sigma v)=Q(v)
\]
for all $v\in V$. The set of all isometries of $Q$ forms a group under composition, called the orthogonal group of $Q$. This group is determined
up to isomorphism by stipulating, as we have done, that $Q$ is non-degenerate and has index $r$. We denote this orthogonal group by
$O^+_{2r}(K)$, or $O^+_{2r}(q)$. 

We call a set of $r$-dimensional totally singular subspaces of $V$ a \emph{partial orthogonal spread} for $V$ and $Q$ if the subspaces intersect
trivially in pairs. Since $V$ contains precisely
\[
 (q^{r-1}+1)(q^r-1)
\]
non-zero singular vectors, it follows that a partial orthogonal spread contains at most $q^{r-1}+1$ subspaces. We refer to a partial
orthogonal spread of size $q^{r-1}+1$ as 
a complete orthogonal spread. 

It is a non-trivial fact
that $V$ contains a complete orthogonal spread if and only if $r$ is even. On the other hand, if $r$ is odd,
the maximum size of a partial orthogonal spread in $V$ is two. 

It is clear that if $\sigma$ is an element of $O^+_{2r}(K)$, and $U$ is a totally singular subspace of $V$, $\sigma(U)$ is also a totally
singular subspace. Thus, $O^+_{2r}(K)$ permutes the maximal ($r$-dimensional) totally singular subspaces, and it is known that this permutation
action is transitive. It is a remarkable fact, proved by Kantor and Williams, that provided $r$ is even, $V$ has a complete orthogonal spread
that consists of a regular orbit of a single maximal totally singular subspace under the action of a cyclic group of order $q^{r-1}+1$.
See Theorem 6.5 and the proof of Theorem 3.3 (i) in \cite{KW}. 

The main purpose of this paper is to construct a partial orthogonal spread of size $2m+1$  in $V$ invariant under the action of the symmetric group
$\Sigma_{2m+1}$, when $\dim V=2^m$, $m\geq 3$ is an integer, and $K=\mathbb{F}_2$. This enables us to prove the following result.
Let $G$ be a finite group with $|G|=2m+\epsilon$, where $\epsilon=0$ or 1. Then if $m\ge 3$, $G$
acts transitively and regularly on a partial orthogonal spread in $V$. We also show that the alternating group $\mathcal{A}_9$ acts
in a natural manner on a complete spread of size 9 defined on a vector space of dimension 8 over $\mathbb{F}_2$.

 \section{Modules for $\Sigma_n$ in characteristic 2}
 
\noindent The irreducible representations of the symmetric group $\Sigma_n$ over a field of characteristic 2 may all be realized
in  $\mathbb{F}_2$. They are labelled by partitions of $n$ where no two parts are equal (so-called 2-regular partitions). 
We let $D^\lambda$ denote the irreducible $\mathbb{F}_2\Sigma_n$-module labelled by the 2-regular partition $\lambda$.

There is one partition $\lambda$ of $n$ for which the corresponding module $D^\lambda$ has remarkable properties, which we wish to describe.
Suppose that $n\geq 3$ is odd, and set $n=2m+1$. Consider the partition $(m+1,m)$ of $n$. It is known that
\[
 \dim D^{(m+1,m)}=2^m. 
\]
Similarly, suppose that $n\geq 4$ is even and write $n=2m$. Then
\[
 \dim D^{(m+1,m-1)}=2^{m-1}
\]
in this case.

Given an $\mathbb{F}_2\Sigma_n$-module, $D$, say, let $D\downarrow_{\Sigma_{n-1}}$ denote its restriction to $\Sigma_{n-1}$.
From what we have described, $D\downarrow_{\Sigma_{n-1}}$ has a composition series consisting of irreducible modules of the form
$D^\mu$, where $\mu$ is a 2-regular partition of $n-1$. 

When $n=2m+1$ is odd, $D^{(m+1,m)}\downarrow_{\Sigma_{2m}}$ is indecomposable
and has a composition series consisting of two copies of $D^{(m+1,m-1)}$. (When $m=1$, $D^{(m+1,m-1)}=D^{(2)}$ is the trivial module
of $\Sigma_2$.) 
When $n=2m$ is even, 
\[
D^{(m+1,m-1)}\downarrow_{\Sigma_{2m-1}}=D^{(m,m-1)}
\]
is irreducible. These results are proved in the main theorem of \cite{Sheth}.

We shall call these irreducible modules $D^{(m+1,m)}$ and $D^{(m+1,m-1)}$ spin modules for $\Sigma_{2m+1}$, $\Sigma_{2m}$,
respectively, over $\mathbb{F}_2$. 

We say that an irreducible $\mathbb{F}_2\Sigma_n$-module $D$ is of quadratic type if there is a non-degenerate $\Sigma_n$- invariant quadratic form
$Q$, say, defined on $D$ (thus $\Sigma_n$ acts as isometries of $Q$). 

We do not know necessary and sufficient conditions on a 2-regular partition $\lambda$ that tell us if $D^\lambda$ is of quadratic type or not,
except in the case that $\lambda$ has only two parts. This  enables us to show that the spin modules are of quadratic type provided
$n\geq 7$. 

We quote the following lemma, proved in \cite{GQ}, Theorem 1.

\begin{lemma} \label{quadratic_type}
 
 Let $\lambda$ be a $2$-regular partition of $n$ into two parts. Then
$D^\lambda$ is not of quadratic type precisely when the smaller part
of $\lambda$ is a power of $2$, say $2^r$, where $r\ge 0$, and 
$n\equiv k\bmod 2^{r+2}$, where $k$ is one of the $2^r$ consecutive integers
$2^{r+1}+2^r-1$, \dots, $2^{r+2}-2$.
 
\end{lemma}

We can now deal with the quadratic type of the spin module.

\begin{lemma} \label{quadratic_type_of_spin}
 
 The spin module of $\Sigma_n$ over $\mathbb{F}_2$ is of quadratic type except when $n=5$ or $6$.
\end{lemma}

\begin{proof}
 
 Suppose that $n=2m$ is even and the spin module $D^{(m+1,m-1)}$ is not of quadratic type. Then Lemma \ref{quadratic_type} implies
 that $m-1=2^r$, for some integer $r\geq 0$, and hence $n=2^{r+1}+2$. We must also have
 \[
  n\equiv k\bmod 2^{r+2},
 \]
 where $k$ is one of $2^{r+1}+2^r-1$, \dots, $2^{r+2}-2$.
 
 It follows certainly that $2^{r+1}$ divides $k-2$, but $2^{r+2}$ does not divide $k-2$. Clearly, then, $k$ cannot equal 2, nor can it equal
 $2^{r+2}+2$, and the possibility that $k-2\geq 3\times 2^{r+1}$ is excluded by the inequality $k\leq 2^{r+2}-2$. Thus we must have
 $k=n=2+2^{r+1}$. 
 
When we take account of the inequality
\[
 2^{r+1}+2^r-1\leq k=2^{r+1}+2,
\]
we deduce that $2^r\leq 3$ and hence $r=0$ or 1. 

It is impossible that $r=0$, since it implies $k=4$, which is inconsistent with the inequality $k\leq 2^{r+2}-2=2$. Thus, $r=1$, $k=6$, 
and it is the case  
that $D^{(4,2)}$ is not of quadratic type. This deals with the case that $n$ is even.

Suppose next that $n=2m+1$ is odd. We have noted above that
\[
 D^{(m+2,m)}\downarrow_{\Sigma_{2m+1}}=D^{(m+1,m)}.
\]
Since $D^{(m+2,m)}$ is of quadratic type for $m\geq 3$ by what we have proved above, we deduce that $D^{(m+1,m)}$ is also of quadratic type for
$m\geq 3$. However, $D^{(3,2)}$ is not of quadratic type, by the criterion of Lemma \ref{quadratic_type}. This completes the analysis.

\end{proof}

\begin{lemma} \label{socle}
 
 Let $Q$ be a non-degenerate $\Sigma_{2m+1}$-invariant quadratic form defined on the spin module $D^{(m+1,m)}$. Let $U$ be the socle of
  $D^{(m+1,m)}\downarrow_{\Sigma_{2m}}$. Then, $U$ is irreducible, has dimension $2^{m-1}$ and is totally
 singular with respect to $Q$. 
 
\end{lemma}

\begin{proof}
 
 We have remarked that $D^{(m+1,m)}\downarrow_{\Sigma_{2m}}$ is indecomposable, having exactly two composition factors, both isomorphic
 to $D^{(m+1,m-1)}$. It follows that the socle $U$ of $D^{(m+1,m)}\downarrow_{\Sigma_{2m}}$ is irreducible, isomorphic to $D^{(m+1,m-1)}$, and
 hence has dimension $2^{m-1}$. 

 Let $f$ be the polarization of $Q$ and let $U^\perp$ be the perpendicular subspace of $U$ with respect to $f$. Since $U$ and $f$ are invariant
 under $\Sigma_{2m}$, $U^\perp$ is also invariant under the group. It follows that $U\cap U^\perp$ is a subspace of $U$ that is also
 $\Sigma_{2m}$-invariant. The irreducibility of $U$ implies that $U\cap U^\perp=0$ or $U$. 
 
 Now if $U\cap U^\perp=0$, the general equality that
 \[
  \dim U+\dim U^\perp=\dim D^{(m+1,m)}
 \]
 implies that $D^{(m+1,m)}\downarrow_{\Sigma_{2m}}$ is the direct sum of the two invariant submodules $U$ and $U^\perp$. This contradicts
 the fact that $U$ is the socle. 
 
 We deduce that $U\cap U^\perp=U$, and hence $U$ is totally isotropic with respect to $f$. This implies that
 \[
  Q(u+w)=Q(u)+Q(w)
 \]
for all elements $u$ and $w$ in $U$, and hence $Q$ is a linear mapping from $U$ to $\mathbb{F}_2$. The subset of all singular
vectors in $U$ is then a subspace of $U$ of codimension at most one in $U$, and this subspace is certainly $\Sigma_{2m}$-invariant. The irreducibility
of $U$ forces the conclusion that $Q$ is identically zero on $U$, and hence $U$ is totally singular.

\end{proof}

\section{Construction of an invariant partial orthogonal spread}

\noindent We have compiled enough information about the spin module to prove the following result relating to a partial orthogonal
spread invariant under symmetric group action.

\begin{theorem} \label{orthogonal_spread}

Suppose that $m\geq 3$. Let $V$ be a vector space of dimension $2^m$ over $\mathbb{F}_2$ that defines the spin module for
$\Sigma_{2m+1}$ and let $U$ denote the socle of $V\downarrow_{\Sigma_{2m}}$.
Let $g_1=1$, \dots, $g_{2m+1}$ be a set of coset representatives for $\Sigma_{2m}$ in $\Sigma_{2m+1}$. Then the $2m+1$ subspaces
$U=g_1U$, $g_2U$, \dots, $g_{2m+1}U$ form a partial orthogonal spread in $V$, permuted by $\Sigma_{2m+1}$ according to its natural action.
 
\end{theorem}

\begin{proof}
 
 We first note that, since $m\geq 3$, $V$ is indeed a module of quadratic type, by Lemma \ref{quadratic_type_of_spin}, and Lemma \ref{socle}
 implies that $U$ is totally singular with respect to the invariant non-degenerate quadratic form.
 Let $g$ be any element of $\Sigma_{2m+1}$ not in $\Sigma_{2m}$.  Consider the subspace $gU$, which is also totally singular. 
 
 We claim that $gU\neq U$. For if $gU=U$, $U$ is invariant under the subgroup of $\Sigma_{2m+1}$ generated by $g$ and  $\Sigma_{2m}$. Since
 $\Sigma_{2m}$ is a maximal subgroup of $\Sigma_{2m+1}$, this subgroup is all of $\Sigma_{2m+1}$. But as $V$ is irreducible for $\Sigma_{2m+1}$,
 $U$ cannot be invariant under $\Sigma_{2m+1}$. We deduce that $gU\neq U$, as claimed.
 
 Consider now the subspace $U\cap gU$, which we have just shown is not the whole of $U$. It is easy to see that $U\cap gU$ is invariant under
 the subgroup $\Sigma_{2m}\cap (g\Sigma_{2m}g^{-1})$ of $\Sigma_{2m}$. If we assume, as we may, that $\Sigma_{2m}$ is the subgroup of
 $\Sigma_{2m+1}$ fixing 1 in the natural representation of $\Sigma_{2m+1}$  on the numbers $\{1, 2, \ldots, 2m+1\}$, we see that
 $\Sigma_{2m}\cap (g\Sigma_{2m}g^{-1})$ is the subgroup fixing 1 and $g(1)$. We may then identify this subgroup unambiguously as $\Sigma_{2m-1}$,
 since the subgroups of $\Sigma_{2m+1}$ fixing 1 and a different number are conjugate.
 
 We know that $U$ affords the spin representation of $\Sigma_{2m}$, and its restriction to $\Sigma_{2m-1}$ is irreducible. Thus the only subspaces of
 $U$ that are invariant under $\Sigma_{2m-1}$ are $U$ and 0. We deduce that $U\cap gU=0$, as required, and the rest of the theorem
 follows  from this argument.
\end{proof}

\begin{corollary} \label{odd_order}
 Let $G$ be a finite group of order at least $6$. Then $G$ acts in a regular transitive manner on a partial orthogonal spread of size
 $|G|$ defined on a quadratic space of dimension $2^m$ over $\mathbb{F}_2$, where $m$ is the integer part of $|G|/2$. 
 
\end{corollary}

\begin{proof}

Suppose first that $|G|=2m+1$, where $m\ge 3$.  
 We may embed $G$ into $\Sigma_{2m+1}$ by means of its regular representation. Then $G$ permutes the subspaces in 
 the partial orthogonal spread described in Theorem \ref{orthogonal_spread} in a regular transitive manner, as required.
 
 Suppose next that $|G|=2m$ is even, with $m\geq 3$. We may embed $G$ into $\Sigma_{2m+1}$ in such a way that $G$ fixes one point
 and permutes the remaining $2m$ points regularly. Then, in the action on the partial orthogonal spread described in Theorem \ref{orthogonal_spread},
 $G$ clearly fixes one subspace and permutes the other $2m$ subspaces regularly. 
 \end{proof}
 
 \noindent\textbf{Note.} If $|G|=3$, 4 or 5, we may embed $G$ into $\Sigma_7$ and then show that $G$ acts transitively on a partial
 orthogonal spread of size $|G|$ defined on a quadratic space of dimension $8$ over $\mathbb{F}_2$. However, a group of order 4 or 5 does not act
 transitively on a partial orthogonal spread of size 4 or 5 on a space of dimension 4 over $\mathbb{F}_2$, since a complete spread only contains
 three subspaces in such a case. A similar remark holds for a group of order 3 acting on a two-dimensional space over $\mathbb{F}_2$.

 \section{Extension of the partial spread when $m\equiv 3 \bmod 4$}
 
 \noindent When we take into account the influence of the alternating subgroup $\mathcal{A}_{2m+1}$ of $\Sigma_{2m+1}$, we shall show
 that if $m\equiv 3\bmod 4$, the partial orthogonal spread of size $2m+1$ just described can be extended by two more maximal totally singular subspaces
 to give a partial orthogonal spread of size $2m+3$. The two additional subspaces are invariant under the alternating group $\mathcal{A}_{2m+1}$ 
 and are interchanged by any odd permutation in $\Sigma_{2m+1}$.
 
 In order to find these additional subspaces, it is necessary to describe how $\mathcal{A}_{2m+1}$ acts on the spin module. The results
 we need to know are quite sensitive to properties of the integer $m$, and require careful
 explanation. 
 
 \begin{lemma} \label{splitting}
  
  The spin module $D^{(m+1,m)}$ of $\Sigma_{2m+1}$ over $\mathbb{F}_2$ splits as a direct sum of two non-isomorphic irreducible
  $\mathbb{F}_2\mathcal{A}_{2m+1}$-modules if $m\equiv 0\bmod 4$ or if $m\equiv 3\bmod 4$. The two $\mathcal{A}_{2m+1}$-modules
  are conjugate under the action of $\Sigma_{2m+1}$, as described by Clifford's theorem.
 \end{lemma}

 \begin{proof}
  
  This follows from Theorem 6.1 of \cite{Ben}.
 \end{proof}

 The next question we need to address is whether or not the two irreducible $\mathbb{F}_2\mathcal{A}_{2m+1}$-modules described in
 Lemma \ref{splitting} are self-dual. Here again, the answer is not obvious, and depends on the residue of $m$ modulo 4. 
 It seems that to establish what we want to know, we must invoke an alternative construction of the spin module over $\mathbb{F}_2$.

Let $\Gamma_{2m+1}$ denote either of the two non-isomorphic double covers of $\Sigma_{2m+1}$. The commutator subgroup of $\Gamma_{2m+1}$
has index 2 in $\Gamma_{2m+1}$ and is a double cover of $\mathcal{A}_{2m+1}$, which we shall denote by $\tilde{\mathcal{A}}_{2m+1}$. 
Since $\tilde{\mathcal{A}}_{2m+1}$ is an extension of a central subgroup of order 2 by $\mathcal{A}_{2m+1}$, given any element of odd order in 
$\mathcal{A}_{2m+1}$, there is a unique element of the same order that projects onto it under the canonical homomorphism from 
$\tilde{\mathcal{A}}_{2m+1}$ 
onto $\mathcal{A}_{2m+1}$. We shall refer to this element of $\tilde{\mathcal{A}}_{2m+1}$ as the canonical inverse image of the given element
of odd order in $\mathcal{A}_{2m+1}$. 

$\Gamma_{2m+1}$ has a faithful irreducible complex representation of degree $2^m$, known as the basic spin representation (it is an example
of a so-called projective representation of $\Sigma_{2m+1}$). Let $\theta$ denote the character
of the basic spin representation. Schur shows in \cite{Schur}, Formula VII*, p. 205, that $\theta$ is rational-valued. Furthermore, it follows 
Theorem 7.7 of \cite{W} that $\theta$ defines an absolutely irreducible Brauer character modulo the prime 2. Corollary 9.4 of Chapter IV of
\cite{F} implies then that $\theta$ has Schur index one over the field $\mathbb{Q}_2$ of 2-adic numbers. Thus, since
$\theta$ certainly takes values in $\mathbb{Q}_2$, we deduce that the basic spin representation may be realized over
$\mathbb{Q}_2$.

Let $\mathbb{Z}_2$ denote the ring of 2-adic integers in $\mathbb{Q}_2$. $R$ is a principal ideal domain and it
follows that there is a $\Gamma_{2m+1}$-invariant $\mathbb{Z}_2$-lattice $L$, say, of rank $2^m$ which affords the basic spin representation.
The quotient $L/2L$ is then a vector space, $\overline{L}$, say, of dimension $2^m$ over $\mathbb{F}_2$. Since the central involution
of $\Gamma_{2m+1}$ acts as $-I$ on $L$, this involution acts trivially on $\overline{L}$, and thus $\overline{L}$ is naturally an
$\mathbb{F}_2\Sigma_{2m+1}$-module, which it turns out is isomorphic to the spin module $D^{(m+1,m)}$ we have been considering. 

Working over the algebraic closure of $\mathbb{Q}_2$, the basic spin module is reducible on restriction to $\tilde{\mathcal{A}}_{2m+1}$. This splitting
does not necessarily occur over $\mathbb{Q}_2$, since we need a square root of $(-1)^m(2m+1)$ for it to take place.
We refer to \cite{Schur}, Formula VII*, p. 205, for this theory. Schur shows that $\theta$ splits into two different irreducible characters
of $\tilde{\mathcal{A}}_{2m+1}$,
$\theta_1$ and $\theta_2$, say. These characters $\theta_1$ and $\theta_2$ are real-valued if and only if
$m$ is even. Furthermore, $\theta_1$ and $\theta_2$ differ on the canonical inverse image of 
a $2m+1$-cycle. In particular,
if $m$ is odd, $\theta_1$ and $\theta_2$ take non-real values on the canonical inverse image of a $2m+1$-cycle.

Now $\theta$ restricted to elements of odd order is the Brauer character of the spin module $D^{(m+1,m)}$ of $\Sigma_{2m+1}$. Since
we know that $D^{(m+1,m)}$ is reducible on restriction to $\mathcal{A}_{2m+1}$, the Brauer characters of the irreducible constituents are
$\theta_1$ and $\theta_2$, again restricted to elements of odd order. Finally, since $\theta_1$ and $\theta_2$ are not real-valued on the canonical 
inverse
image of a $2m+1$-cycle, the Brauer characters defined by $\theta_1$ and $\theta_2$ are not 
real-valued, and consequently, the two irreducible constituents of $D^{(m+1,m)}\downarrow_{\mathcal{A}_{2m+1}}$ are not self-dual.

When we collect the information we have derived from the work of Schur, we have proved the following important result.

\begin{lemma} \label{not_self_dual}
 The spin module $D^{m+1,m}$ of $\Sigma_{2m+1}$ over $\mathbb{F}_2$ splits as a direct sum of two non-isomorphic irreducible
  $\mathbb{F}_2\mathcal{A}_{2m+1}$-modules if $m\equiv 3\bmod 4$. The two 
  $\mathbb{F}_2\mathcal{A}_{2m+1}$-modules are not self-dual
  in this case.
 
\end{lemma}

We proceed to extend the partial orthogonal spread when $m\equiv 3\bmod 4$.

\begin{theorem} \label{extended_spread}
 Suppose that $m\equiv 3\bmod 4$. Let $V$ be a vector space of dimension $2^m$ over $\mathbb{F}_2$ that defines the spin module for
$\Sigma_{2m+1}$ and let $U$ denote the socle of $V\downarrow_{\Sigma_{2m}}$. Let $V\downarrow_{\mathcal{A}_{2m+1}}=U_1\oplus U_2$, where $U_1$ and $U_2$
are non-isomorphic irreducible 
$\mathbb{F}_2\mathcal{A}_{2m+1}$-modules. Then $U_1$ and $U_2$ are both totally singular. 

Furthermore, let $g_1=1$, \dots, $g_{2m+1}$ be a set of coset representatives for $\Sigma_{2m}$ in $\Sigma_{2m+1}$. Then 
$U=g_1U$, $g_2U$, \dots, $g_{2m+1}U$, $U_1$ and $U_2$ form a partial orthogonal spread in $V$ consisting of $2m+3$ subspaces.

\end{theorem}

\begin{proof}
 
 We note that all $2m+3$ subspaces described above have dimension $2^{m-1}$. We also proved in Lemma \ref{not_self_dual} that
 $U_1$ and $U_2$ are not self-dual. Since they are both irreducible under the action of
 $\mathcal{A}_{2m+1}$, it follows in a straightforward way (as in the proof of Lemma
 \ref{socle}) that $U_1$ and $U_2$ are both totally singular. 
 
 We claim that $U\neq U_1$. For, $U$ is invariant under $\Sigma_{2m}$, whereas $U_1$ is invariant under $\mathcal{A}_{2m+1}$ but not under
 $\Sigma_{2m+1}$. Let $g$ be any odd permutation in $\Sigma_{2m}$. Then $g\not\in \mathcal{A}_{2m+1}$ also, and hence $gU_1=U_2$. Now if
 $U=U_1$, then $U=gU=gU_1=U_2$. This is clearly a contradiction.
 
 We deduce that $U\cap U_1$ is a proper subspace of $U$, since $\dim U=\dim U_1$. $U\cap U_1$  is also invariant under $\Sigma_{2m}\cap
 \mathcal{A}_{2m+1}=\mathcal{A}_{2m}$. We also know from the previous section of this paper that $U$ is an irreducible 
 $\mathbb{F}_2\Sigma_{2m}$-module, isomorphic to the spin module $D^{(m+1,m-1)}$. Since $m$ is odd by assumption, 
 $D^{(m+1,m-1)}\downarrow_{\mathcal{A}_{2m}}$ is irreducible, by Theorem 1.1 of \cite{Ben}. Thus the only $\mathcal{A}_{2m}$-submodules
 of $U$ are $U$ and 0. Since $U\cap U_1$ is an $\mathcal{A}_{2m}$-submodule, and not equal to $U$, it must be 0. A similar argument
 proves that $U\cap U_2=0$ also.
 
 Finally, let $h$ be any element of $\Sigma_{2m+1}$ not in $\Sigma_{2m}$, and consider $(hU)\cap U_1$. Since $h^{-1}U_1=U_1$ or $U_2$, we see that
 \[
  (hU)\cap U_1=h(U\cap h^{-1}U_1)=0,
   \]
since we have proved that $U\cap U_1=U\cap U_2=0$. Similarly, $(hU)\cap U_2=0$ also. This completes the proof.
\end{proof}

This argument shows, for example, that the partial orthogonal spread of seven subspaces in an 8-dimensional space over $\mathbb{F}_2$, 
invariant under the action of $\Sigma_7$, can be extended to a complete spread of nine subspaces, also
invariant under  $\Sigma_7$.

\section{Action of $A_9$ on a complete spread in $8$ dimensions}

\noindent We have just shown that $\Sigma_7$ acts on a complete spread of nine subspaces in an 8-dimensional space over $\mathbb{F}_2$.
We intend to give another explanation of this fact by showing that $\mathcal{A}_{9}$ acts on a
complete spread of nine subspaces in an 8-dimensional space over $\mathbb{F}_2$ and then observing
 that $\Sigma_7$ is a subgroup of $\mathcal{A}_{9}$. 
 
 We take as our starting point the data that $\mathcal{A}_{9}$ has three inequivalent irreducible representations of degree 8 over $\mathbb{F}_2$.
 We need to exclude one of these representations from consideration, and this is the so-called deleted permutation module, which arises
 from the natural permutation action of $\mathcal{A}_{9}$ on nine points. The restriction of the deleted permutation module to
 $\mathcal{A}_{8}$ has a composition series consisting of an irreducible module of dimension 6 and two copies of the trivial module,
and this is not what we want.

The modules that we require arise from the restriction of the 16-dimensional spin module $D^{(5,4)}$ of $\Sigma_9$ to $\mathcal{A}_{9}$.
We remarked already that Theorem 6.1 of \cite{Ben} implies that $D^{(5,4)}\downarrow_{\mathcal{A}_{9}}$ is the direct sum of two non-isomorphic
$\mathbb{F}_2\mathcal{A}_9$-modules. These two 8-dimensional modules are both self-dual. By way of proof, albeit not a self-contained one,
we can refer to p.85 of \cite{atlas}, where we see that all three inequivalent irreducible representations of $\mathcal{A}_9$ of degree
8 in characteristic 2 support $\mathcal{A}_9$-invariant quadratic forms.

We can now proceed to fashion this information into a statement about the action of $\mathcal{A}_9$ on  a complete orthogonal
spread. 

\begin{theorem} \label{A_9-action}

Let $V$ be an $\mathbb{F}_2\mathcal{A}_9$-module of dimension $8$ defined as an irreducible constituent of the restriction of the
spin module $D^{(5,4)}$ of $\Sigma_9$. Then $V$ is a module of quadratic type. Let $U$ be an irreducible $\mathbb{F}_2\mathcal{A}_8$-submodule
of $V$. Then $U$ is $4$-dimensional and is totally singular. Moreover, if $g_1$, \dots, $g_9$ are a set of coset representatives for
$\mathcal{A}_{8}$ in $\mathcal{A}_{9}$, the nine subspaces $g_iU$, $1\leq i\leq 9$, form a complete orthogonal spread in $V$.
 
\end{theorem}

\begin{proof}

 We first note that we can take the Brauer character of $\mathcal{A}_{9}$ acting on $V$ to be that denoted by $\phi_3$ in the table found
 on p.85 of \cite{atlas}. (The character $\phi_4$ has the same properties as $\phi_3$, and is conjugate to $\phi_3$ under the action of 
 $\Sigma_9$.) As we noted above, $V$ is of quadratic type. (This also follows from the fact that $D^{(5,4)}$  is of quadratic type
 and $V$ is self-dual, since its Brauer character is real-valued.)
 
 Reference to the table on p.48 of \cite{atlas} shows that the restriction of $\phi_3$ to $\mathcal{A}_{8}$ consists of two different
 irreducible Brauer characters of degree 4, one being the complex conjugate of the other. Now it is a fact that $V$ is reducible as an
 $\mathbb{F}_2\mathcal{A}_8$-module. To see this, we note that $\mathcal{A}_{8}$ is isomorphic to the general linear group
 GL$_4(\mathbb{F}_2)$. The Brauer characters that occur in the restriction of $\phi_3$ to $\mathcal{A}_{8}$  are those of the natural
 4-dimensional module for GL$_4(\mathbb{F}_2)$ over $\mathbb{F}_2$ and its contragredient (or dual). If $V$ were irreducible as an
 $\mathbb{F}_2\mathcal{A}_8$-module, it would follow that the two non-isomorphic irreducible modules for GL$_4(\mathbb{F}_2)$ were Galois-conjugate
 over $\mathbb{F}_4$, which is not the case, as they are defined over $\mathbb{F}_2$.
 
 This argument establishes that $U$ is 4-dimensional and furthermore, it is not self-dual, since its Brauer character is not real-valued.
 This then implies that $U$ is totally singular. 
 
 To complete the proof, we imitate the proof of Theorem \ref{orthogonal_spread}.
 
 Let $g$ be any element of $\mathcal{A}_{9}$ not in $\mathcal{A}_{8}$.  Consider the subspace $gU$, which is also totally singular. 
  We claim that $gU\neq U$. For if $gU=U$, $U$ is invariant under the subgroup of $\mathcal{A}_{9}$ generated by $g$ and  $\mathcal{A}_{8}$. Since
 $\mathcal{A}_{8}$ is a maximal subgroup of $\mathcal{A}_{9}$, this subgroup is all of $\mathcal{A}_{9}$. But as $V$ is irreducible for 
 $\mathcal{A}_{9}$,
 $U$ cannot be invariant under $\mathcal{A}_{9}$. We deduce that $gU\neq U$, as claimed.
 
 Consider now the subspace $U\cap gU$, which we have just shown is not the whole of $U$. It is easy to see that $U\cap gU$ is invariant under
 the subgroup $\mathcal{A}_{8}\cap (g\mathcal{A}_{8}g^{-1})$ of $\mathcal{A}_{9}$. We may take $\mathcal{A}_{8}$ to be the subgroup of
 $\mathcal{A}_{9}$ fixing 1 in the natural representation of $\mathcal{A}_{9}$  on the numbers $\{1, 2, \ldots, 9\}$. Then we see that
 $\mathcal{A}_{8}\cap (g\mathcal{A}_{8}g^{-1})$ is the subgroup fixing 1 and $g(1)$, which is isomorphic to $\mathcal{A}_{7}$.
  
 We know that $U$ affords an irreducible representation of $\mathcal{A}_{8}$, and its restriction to $\mathcal{A}_{7}$ is irreducible.
 Thus the only subspaces of
 $U$ that are invariant under $\mathcal{A}_{7}$ are $U$ and 0. We deduce that $U\cap gU=0$, as required, and the rest of the theorem
 follows  from this argument.
\end{proof}

We note that any group $G$ of order 9 may be embedded into $\mathcal{A}_{9}$ by its regular representation. It follows that $G$ 
acts in a regular transitive manner on a complete orthogonal spread of size
 9 defined on a quadratic space of dimension $8$ over $\mathbb{F}_2$. This is a consequence of the theorem of Kantor and Williams already 
 described when $G$ is cyclic, but seems to be a new observation when $G$ is elementary abelian. Similarly, any 
 non-cyclic group $G$ of order 8 may be embedded into $\mathcal{A}_{8}$ by its regular representation, and then into
 $\mathcal{A}_{9}$. It follows that $G$ 
acts in a regular transitive manner on a partial orthogonal spread of size
 8 defined on a quadratic space of dimension $8$ over $\mathbb{F}_2$.
 
 There are 135 non-zero singular vectors in the 8-dimensional quadratic space of index 4  over $\mathbb{F}_2$. These vectors
 are permuted transitively by $\mathcal{A}_{9}$. The action is imprimitive, there being nine blocks of imprimitivity, namely, the non-zero vectors
 in each of the 4-dimensional subspaces that constitute the invariant complete spread. The stabilizer of a block is isomorphic to $\mathcal{A}_{8}$, 
 and it acts
 doubly transitively on the 15 non-zero vectors in the block.

\end{document}